\documentclass{stylefile}
\usepackage{bm}
\usepackage{scalerel}
\usepackage{enumerate}
\usepackage[shortlabels]{enumitem}

\usepackage{calrsfs}
\DeclareMathAlphabet{\pazocal}{OMS}{zplm}{m}{n}

\usepackage{diagrams}

\usepackage[version=3]{mhchem}
\usepackage{accents}
\usepackage{bbm}
\usepackage[outline]{contour}
\usepackage{color}
\definecolor{olaj}{RGB}{70,88,104}
\definecolor{zold}{RGB}{151,151,49}
\definecolor{zold}{RGB}{70,123,148}
\definecolor{bronz}{RGB}{123,79,76}
\definecolor{tortfeher}{RGB}{229,226,203}
\definecolor{tortsarga}{RGB}{219,183,133}

\makeatletter
\newcommand*\bigcdot{\mathpalette\bigcdot@{.5}}
\newcommand*\bigcdot@[2]{\mathbin{\vcenter{\hbox{\scalebox{#2}{$\m@th#1\bullet$}}}}}
\makeatother

\usepackage{comment}

\definecolor{mydarkorange}{RGB}{0,0,0}
\definecolor{gold}{rgb}{0,0,0}
\definecolor{grey}{RGB}{0,0,0}
\definecolor{myorange}{RGB}{0,0,0}
\definecolor{mydarkorange}{RGB}{0,0,0}
\definecolor{mylightblue}{RGB}{0,0,0}
\definecolor{myyellow}{RGB}{0,0,0}
\definecolor{purple}{RGB}{0,0,0}
\definecolor{myblue}{RGB}{0,0,0}
\definecolor{mygreen}{RGB}{0,0,0}

\definecolor{brown}{RGB}{153,88,43}

\usepackage{graphicx}
\usepackage[export]{adjustbox}

\usepackage{amsmath}
\usepackage{mathtools}
\usepackage{amssymb}

\newtheorem{theorem}{Theorem}[section]
\newtheorem{lemma}[theorem]{Lemma}

\theoremstyle{definition}
\newtheorem{definition}[theorem]{Definition}

\newtheorem{proposition}[theorem]{Proposition}
\newtheorem{corollary}[theorem]{Corollary}

\theoremstyle{remark}
\newtheorem{remark}[theorem]{Remark}

\numberwithin{equation}{section}

\newcommand*\cleartoleftpage{%
  \clearpage
  \ifodd\value{page}\hbox{}\newpage\fi
}

\newcommand{\iksz}           {^{\scaleto{(\mathbf X)}{4pt}}}

\newcommand{\ipsz}           {^{\scaleto{(\mathbf Y)}{4pt}}}

\newcommand{\komp}      {{^\prime}}

\newcommand{\kompM}[1]{^{{\prime^{\mkern-4mu^{_{#1}}}}}}
\newcommand{\nega}      [1] {{#1}\komp}
\newcommand{\negaM}      [2] {{#2}\kompM{#1}}
\newcommand{\negaiksz}      [1] {\negaM{{\scaleto{(\mathbf X)}{3,5pt}}}{#1}}
\newcommand{\negaipsz}      [1] {\negaM{{\scaleto{(\mathbf Y)}{3,5pt}}}{#1}}

\newcommand{\te}{{\mathbin{*\mkern-9mu \circ}}}

\newcommand{\tepont}{\mathbin{\circ\mkern-7mu \cdot\mkern2mu}}
\newcommand{\teiksz}{\mathbin{\cdot\iksz }}

\newcommand{\teipsz}{\mathbin{\cdot\ipsz }}

\newcommand{\ite}[1]{\mathbin{\rightarrow_{#1}}}

\newcommand{\g}                 [2] {{#1} \mathbin{\te} {#2}}

\newcommand{\gteM}       [3] {{#2} \mathbin{\te_{#1}} {#3}}

\newcommand{\gpont}        [2] {{#1} \tepont {#2}}
\newcommand{\giksz}        [2] {{#1} \teiksz {#2}}
\newcommand{\gipsz}        [2] {{#1} \teipsz {#2}}

\newcommand{\res}               [3] {{#2}\mathbin{\ite{#1}}{#3}}
\newcommand{\resiksz}               [2] {{#1}\mathbin{\ite{}\iksz }{#2}}
\newcommand{\resipsz}               [2] {{#1}\mathbin{\ite{}\ipsz }{#2}}

\newcommand{\lex}                                        {\overset{\rightarrow}{\times}}

\newcommand{\Twoheadrightarrow}               {\rightarrow \mkern-12.25mu \rightharpoonup}

\newcommand{\plexII}               {\overset{\Twoheadrightarrow}{\times}}

\newcommand{\PLPII}            [2] {{#1}\plexII{{#2}}}

\begin{document}



\title{Densification in classes of involutive commutative residuated lattices}


\begin{abstract}

The representation theorem for odd or even involutive FL$_e$-chains by bunches of layer groups, as discussed in \cite{JenRepr2020}, is redefined to demonstrate a more straightforward constructional relationship between odd or even involutive FL$_e$-chains and bunches of layer groups, bypassing the intermediary stage of layer algebras.
By leveraging this redefined theorem, it is demonstrated that both the variety of semilinear odd involutive FL$_e$-algebras and its idempotent symmetric subvariety admits densification.
Ultimately, employing the algebraic techniques introduced in \cite{JS_Hahn}, the proof of the strong standard completeness of Involutive Uninorm Logic with Fixed Point ($\mathbf{IUL}^{fp}$) is established, thus strengthening the main result of \cite{JS_FSSC}.

\end{abstract}

\author{S\'andor Jenei}


\address{Institute of Mathematics and Informatics, Eszterh\'azi K\'aroly Catholic University, Hungary,
and 
Institute of Mathematics and Informatics, University of P\'ecs, Hungary}



\date{}

\maketitle


\section{Introduction and preliminaries}

Varieties of FL-algebras (also known as pointed residuated lattices) serve as algebraic counterparts of substructural logics \cite{gjko}. This area significantly intersects with the class of mathematical fuzzy logics \cite{MetMon 2007}.
Currently, there exists less comprehensive knowledge regarding substructural logics without weakening compared to those involving weakening. The aim of this paper is to take a small step into exploring this uncharted territory.
We shall prove the strong standard completeness of Involutive Uninorm Logic with Fixed Point ($\mathbf{IUL}^{fp}$).

Standard algebras for mathematical fuzzy logics are the ones from the corresponding variety which universe is the real unit interval $[0,1]$.
A mathematical fuzzy logic $L$ (or the variety which is its equivalent algebraic counterpart) enjoys strong standard completeness if the following conditions are equivalent for each formula $\varphi$ and theory $T$ in $L$: 
(1) $T\vDash_L \varphi$, (2) for each standard L-algebra $\mathbf A$ and each $\mathbf A$-model $e$ of $T$, $e$ is an $\mathbf A$-model of $\varphi$.
There are two weaker alternatives for defining standard completeness. The same definition but confining to finite theories yields the definition of finite strong standard completeness, whereas by setting $T=\emptyset$ one obtains the definition of (weak) standard completeness.

Densification is a key component to prove standard completeness of mathematical fuzzy logics.
The idea 
of proving strong standard completeness of a mathematical fuzzy logic 
via densification of countable chains followed by a subsequent Dedekind-MacNeille completion has been introduced in \cite{JMstcompl}.
Since then this method has become folklore and has been extensively applied for proving strong standard completeness of a whole lot of substructural fuzzy logics in the literature.
To exhibit one more application, 
we prove here the strong standard completeness of $\mathbf{IUL}^{fp}$. 
The finite strong standard completeness of this logic has already been settled in \cite{JS_FSSC}. Its proof relied on a representation theorem of those odd involutive FL$_e$-chains which has only finitely many positive idempotent elements \cite{JS_Hahn,Jenei_Hahn_err}, and that representation theorem makes use of finite iteration of the so-called partial sublex product construction.
Such algebras are finitely generated, and a usual way to prove finite strong standard completeness is to embed finitely generated FL$_e$-chains into standard ones. In fact, the finite iteration in the representation theorem guided the definition of a finite sequence of consecutive embeddings in \cite{JS_FSSC}.
The strong standard completeness of the same logic via embedding arbitrary (not only finitely generated) countable chains into standard ones in the present paper is not only a stronger result but its proof is substantially simpler, too. 

An original decomposition method along with the related (re)construction has been introduced in \cite{JenRepr2020} for the class of odd or even involutive FL$_e$-chains.
This methodological contribution aimed to address the intricate structures within this class of residuated chains, providing a novel framework for analysis and manipulation.
The fundamental concept involved partitioning the algebra utilizing its local unit function $x\mapsto\res{}{x}{x}$ into a direct system comprising potentially simpler algebras. These were indexed by the positive idempotent elements of the original algebra, with transitions within the direct system defined by multiplication with a positive idempotent element. The reconstruction method used, partially akin to P\l{}onka sums.
It is termed layer algebra decomposition.
This concept has been recently implemented to structurally describe various classes of residuated lattices, extending beyond its initial application. These encompass finite commutative, idempotent, and involutive residuated lattices, \cite{JTV2021}, finite involutive po-semilattices \cite{JiSu22}, and locally integral involutive po-monoids and semirings \cite{GFoth}.
In these classes the layer algebras are \lq\lq nice\rq\rq.
In \cite{JenRepr2020}, however, the layer algebras exhibit only a moderate improvement in terms of their \lq\lq niceness\rq\rq\ compared to the original algebra.
Consequently, an additional step was necessary, introducing a phase involving the construction (and subsequent reconstruction) of layer groups from layer algebras.
The layer algebra decomposition, coupled with the transformation of layer algebras into layer groups and the reverse construction, culminated in a rather intricate process in \cite[Theorem 8.1]{JenRepr2020}. Ultimately, this process led to a representation of odd or even involutive FL$_e$-chains by bunches of layer groups.
This representation has been lifted to a categorical equivalence in \cite{JScategorical}, and has proven to be a powerful weaponry to prove amalgamation results in classes of involutive FL$_e$-algebras which are neither integral, nor divisible, nor idempotent \cite{JSamalg}.
In Theorem~\ref{BUNCH_X} below, we establish a more direct constructional relationship between an (odd or even involutive) FL$_e$-chain and its corresponding layer groups, bypassing the intermediary stage of layer algebras.
This formulation not only simplifies the formalism but also enhances its adaptability in tackling subsequent mathematical problems. 
As a direct application, we prove the densification property for two varieties
of odd involutive FL$_e$-algebras, namely for the semilinear one and its idempotent symmetric subvariety.
Ultimately, we prove the strong standard completeness of Involutive Uninorm Logic with Fixed Point ($\mathbf{IUL}^{fp}$) in the algebraic style of \cite{JS_Hahn}, thereby surpassing the main result of \cite{JS_FSSC}.

\section{Preliminaries}
Algebras will be denoted by bold capital letters, their underlying sets by the same regular letter unless otherwise stated.
Let $\mathbf X=(X, \leq)$ be a 
poset.
For $x\in X$ define the upper neighbor $x_\uparrow$ of $x$ to be the unique cover of $x$ if such exists, and $x$ otherwise.
Define $x_\downarrow$ dually.
A partially ordered algebra will be called {\em discretely ordered} if for any element $x$, $x_\downarrow<x<x_\uparrow$ holds.
An {\em FL$_e$-algebra}\footnote{Other terminologies are pointed commutative residuated lattices or pointed commutative residuated lattice-ordered monoids.} 
is a structure $\mathbf X=( X, \leq, \cdot,\ite{}, t, f )$ such that 
$(X, \leq )$ is a lattice, $( X, \leq,\cdot,t)$ is a commutative residuated monoid, 
and $f$ is an arbitrary constant, called the {\em falsum} constant.
If the multiplication operation is clear from the context, we write ${x}{y}$ for ${x}\cdot{y}$, as usual.
{\em Commutative residuated lattices} are the $f$-free reducts of FL$_e$-algebras. 
Being residuated means that there exists a binary operation $\ite{}$,
called the residual operation of $\cdot$, such that ${x}{y}\leq z$ if and only if $\res{}{x}{z}\geq y$. This equivalence is called adjointness condition, ($\cdot,\ite{}$) is called an adjoint pair. Equivalently, for any $x,z$, the set $\{v\ | \ {x}{v}\leq z\}$ has its greatest element, and $\res{}{x}{z}$, the residuum of $x$ and $z$, is defined as this element: $\res{}{x}{z}:=\max\{v\ | \ {x}{v}\leq z\}$; this is called the residuation condition. Being residuated implies that $\cdot$ is lattice ordered, that is $\cdot$ distributes over join.
One defines the {\em residual complement operation} by $\nega{x}=\res{}{x}{f}$ and calls an FL$_e$-algebra {\em involutive} if $\nega{(\nega{x})}=x$ holds.
In the involutive case $\res{}{x}{y}=\nega{({x}{\nega{y}})}$ holds. 
Call an element $x\geq t$ {\em positive}. 
An involutive FL$_e$-algebra is called {\em odd} if the residual complement operation leaves the unit element fixed, that is, $\nega{t}=t$, and {\em even} if the following (two) quasi-identities hold: $x<t$ $\Leftrightarrow$ $x\leq f$. 
The former condition is equivalent to $f=t$, while 
the latter quasi-identities are equivalent to assuming that 
$f$ 
is the lower cover of 
$t$ (and $t$ 
is the upper cover of 
$f$)
if chains are considered, that is, when the order is total. 
Let us denote the class of odd and the class of even involutive FL$_e$-chains by
$\mathfrak I^{\mathfrak c}_{\mathfrak 0}$
and
$\mathfrak I^{\mathfrak c}_{\mathfrak 1}$,
respectively,
and their union by 
$\mathfrak I^{\mathfrak c}_{\mathfrak 0\mathfrak 1}$\footnote{Odd and even involutive FL$_e$-chains are also called rank $0$ and rank $1$ involutive FL$_e$-chains, respectively, hence the notation.}.

\medskip
A directed partially ordered set is a partially ordered set such that every pair of elements has an upper bound.
\begin{definition}
Let $\langle \kappa,\leq \rangle$ be a directed partially ordered set.
Let $\{\mathbf A_i\in\mathfrak U : i\in\kappa\}$ be a family of algebras of the same type and $\varsigma_{i\to j}$ be a homomorphism\footnote{Homomorphisms are understood according to the corresponding setting. We shall call them the transitions of the direct system.} for every $i,j\in\kappa$, $i\leq j$ with the following properties:
\begin{enumerate}[(D1)]
\item\label{IDes}
$\varsigma_{i\to i}$ is the identity of $\mathbf A_i$, and
\item\label{Kompooot}
$\varsigma_{i\to k}=\varsigma_{j\to k}\circ \varsigma_{i\to j}$ for all $i\leq j\leq k$.
\end{enumerate}
\noindent
Then $\langle \mathbf A_i,\varsigma_{i\to j} \rangle$ is called a direct system of algebras in $\mathfrak U$ over $\kappa$. 
\end{definition}
\begin{definition}\label{homodirsyst} 
Let $\pazocal A=\langle \mathbf A_i,f_{i\to j} \rangle_{\langle\alpha,\leq_\alpha\rangle}$ 
and 
$\pazocal B=\langle \mathbf B_i,g_{i\to j} \rangle_{\langle\beta,\leq_\beta\rangle}$
be two direct systems from the same class $\mathfrak U$ of algebraic systems. 
By a (direct system) homomorphism $\Phi:\pazocal A\to\pazocal B$ 
we mean
a system of $\mathfrak U$-homomorphisms $\Phi=\{\Phi_i:A_i\to B_{\iota_o(i)} \ | \  i\in\alpha\}$
such that $\iota_o:\alpha\to\beta$ is an $o$-embedding
and
for every $i,j\in\alpha$, $i\leq j$ the diagram in Fig.\,\ref{HomoM}
\begin{figure}[ht]
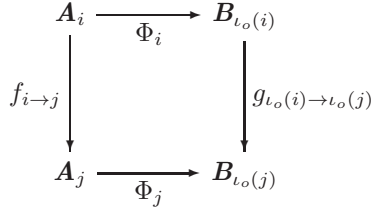

\begin{diagram}
\textbf{\textit{A$_i$}} & \rTo_{\Phi_i} & \textbf{\textit{B$_{\iota_o(i)}$}} \\
\dTo^{f_{i\to j}} & & \dTo_{g_{{\iota_o(i)}\to{\iota_o(j)}}} \\
\textbf{\textit{A$_j$}} & \rTo_{\Phi_j} & \textbf{\textit{B$_{\iota_o(j)}$}} \\
\end{diagram}
\caption{Transitions commute with $\mathfrak U$-homomorphisms}
\label{HomoM}
\end{figure}
commutes.
\end{definition}
As said in the introduction, every odd or even involutive FL$_e$-chain will be represented by a bunch of layer groups in Theorem~\ref{BUNCH_X}.
To this end we need Definition~\ref{DEFbunch}.
A bunch of layer groups is a direct system of abelian $o$-groups enriched with a few additional properties.
In Theorem~\ref{BUNCH_X} we simplify \cite[Theorem 8.1]{JenRepr2020} by exhibiting a more direct link between 
$\mathfrak I^\mathfrak c_{\mathfrak0\mathfrak1}$ 
and the class $\mathfrak B_\mathfrak G$ of bunches of layer groups. 
Even a bijective correspondence can be established between $\mathfrak I^\mathfrak c_{\mathfrak0\mathfrak1}$ and a slightly modified version of $\mathfrak B_\mathfrak G$. 
%
In a bunch of layer groups 
$\langle \textbf{\textit{G$_u$}},\textbf{\textit{H$_u$}}, \varsigma_{u\to v} \rangle_{\langle \kappa_o, \kappa_J, \kappa_I, \leq_\kappa\rangle}$
there are three pairwise disjoint sets $\kappa_o$, $\kappa_J$, and $\kappa_I$, 
their union $\kappa$ is totally ordered by $\leq_\kappa$ such that $\kappa$ has a least element $t$.
It can be seen whether the corresponding  involutive FL$_e$-chain is odd, even with a non-idempotent falsum constant, or 
even with an idempotent falsum constant (let us call it the type of the algebra), by looking at whether $t$ is in $\kappa_o$, in $\kappa_J$, or in $\kappa_I$, respectively. 
\begin{definition}\label{DEFbunch}
A bunch of layer groups
$${\mathcal X}=\langle \bm{G}_u,\bm{H}_u, \varsigma_{u\to v} \rangle_{\langle \kappa_o, \kappa_J, \kappa_I, \leq_\kappa\rangle}$$
is a direct system 
$\langle \bm{G}_u,\varsigma_{u\to v}\rangle_{\langle\kappa, \leq_\kappa\rangle}$ of abelian $o$-groups
over the totally ordered set
$\kappa=
\kappa_o\,\cup\,\kappa_J\,\cup\,\kappa_I$
($\kappa_o$, $\kappa_J$, $\kappa_I$ are pairwise disjoint)
with least element $t$,
$u$ being the unit element of $\bm{G}_u$,
such that
\begin{enumerate}[start=1,label={(G\arabic*)}]
\item \label{(G1)} 
$\kappa_o\subseteq\{t\}$,
\item\label{(G3)}
for $v\in\kappa_I$, $\bm{H}_v\leq\bm{G}_v$
\footnote{$\textbf{\textit{H$_u$}}$'s are indexed by $\kappa_I$ only.}
and for
$u<_\kappa v\in\kappa_I$,
$\varsigma_{u\to v}$ maps into $H_v$,
\item\label{(G2)}
for $u\in\kappa_J$, $\bm{G}_u$ is discrete
and for $v>_\kappa u\in\kappa_J$,
$\varsigma_{u\to v}(u)=\varsigma_{u\to v}(u_{\downarrow_u})$.\footnote{${ }_{\downarrow_u}$ denotes the neighborhood operation in $\bm{G}_v$.}
\end{enumerate}
Call the $\bm{G}_u$'s and the $\bm{H}_u$'s the layer groups and layer subgroups of $\mathcal X$, respectively, call $\langle\kappa,\leq_\kappa\rangle$ the {\em skeleton} of $\mathcal X$, call $\langle \kappa_o, \kappa_J, \kappa_I \rangle$ the {\em partition} of the skeleton, and call $\langle \bm{G}_u, \varsigma_{u\to v} \rangle_{\langle \kappa, \leq_\kappa\rangle}$ the direct system of $\mathcal X$.
\end{definition}

A partially ordered algebra with at least two elements is called dense if for any two distinct comparable elements of it there is a third one strictly in between them. If for some $a<b$ there is no $c$ such that $a<c<b$ then we call the pair $a,b$ a gap.
A class of partially ordered algebras $\mathtt C$ admits densification if every algebra in $\mathtt C$ can be embedded into a dense algebra in $\mathtt C$. 
A variety $\mathtt V$ is called semilinear if every subdirectly irreducible algebra in $\mathtt V$ is a chain. Therefore, referring to the subdirect product representation of Birkhoff, for a semilinear variety, it is sufficient to assume densification for its chains: we say that a variety $\mathtt V$ admits densification if every chain in $\mathtt V$ can be embedded into a dense chain in $\mathtt V$. 
We also say that a class $\mathtt C$ of linearly ordered algebras (not a variety) 
admits densification if every chain in $\mathtt C$ can be embedded into a dense chain in $\mathtt C$. 
Semilinearity of $\mathtt V$ is characteristic to the semantic consequence relation of $\mathtt V$ being the same as the semantic consequence relation of its chains, see e.g.\ \cite{HorcikAlgSem}.

As said in the introduction, strong standard completeness of a logic $L$ can be shown by proving that any countable $L$-chain is embeddable into a standard $L$-chain, and that it can be achieved by densification and a subsequent MacNeille completion. 
Involutive Uninorm Logic with Fixed Point ($\mathbf{IUL}^{fp}$) has been introduced in \cite{MetcThes},
we refer the reader to the introduction of \cite{JS_FSSC} for a motivation for this interesting logic.
Knowing that $\rm{IUL}^{fp}$-chains, that is, non-trivial bounded odd involutive FL$_e$-chains constitute an algebraic semantics of $\mathbf{IUL}^{fp}$, 
we shall prove in this paper that any non-trivial countable, bounded odd involutive FL$_e$-chain 
embeds into an odd involutive FL$_e$-chain over the real unit interval $[0,1]$ such that its top element is mapped into $1$ and its bottom element is mapped into $0$.

\section{A \lq\lq shorter\rq\rq\ bridge between $\mathfrak I^\mathfrak c_{\mathfrak0\mathfrak1}$ and $\mathfrak B_\mathfrak G$}
In the next theorem we reformulate \cite[Theorem~8.1]{JenRepr2020} to better suit to the purpose of the present paper.

\begin{theorem}\label{BUNCH_X}
\begin{enumerate}[label={(A)}]
\item\label{errefere}
Given an odd or an even involutive FL$_e$-chain $\mathbf X=(X,\leq,\cdot,\ite{},t,f)$ with residual complement operation $\komp$,
\begin{equation}\label{IgyNeznekKi}
\mathcal X_{\mathbf X}=\langle \textbf{\textit{G$_u$}},\textbf{\textit{H$_u$}}, \varsigma_{u\to v} \rangle_{\langle \kappa_o, \kappa_J, \kappa_I,\leq_\kappa\rangle}
\ \  
{\mbox{with}}
\ \
\textbf{\textit{G$_u$}} = (G_u,\leq_u,\cdot_u,\ { }^{-1_u},u)
\ \ \ 
(u\in\kappa)
\end{equation}
is bunch of layer groups, called the {\em bunch of layer groups of $\mathbf X$} or the {\em group representation of $\mathbf X$},
where
\begin{equation}\label{IGYleszSKELETON}
\kappa=\{\res{}{x}{x} : x\in X\}=
\{u\geq t : u \mbox{ is idempotent} \} 
\mbox{ is ordered by $\leq$,}
\end{equation}
\begin{equation}
\begin{array}{lll}
\bar\kappa_I
&=&
\{u\in \kappa\setminus\{t\} : \nega{u} \mbox{ is idempotent}\},\\
\bar\kappa_J&= & 
\{u\in \kappa\setminus\{t\} : \nega{u} \mbox{ is not idempotent}\},
\\
\end{array}
\end{equation}
$\kappa_o$, $\kappa_J$, $\kappa_I$ are defined by 
\begin{table}[h]
\begin{center}
\begin{tabular}{l|l|l|lll}
$\kappa_o$ \ \ \ \ \ \ \ \ & $\kappa_J$ & $\kappa_I$ & \\
\hline
\{t\} & $\bar\kappa_J$ & $\bar\kappa_I$ & if $\mathbf X$ is odd\\
\hline
$\emptyset$ & $\bar\kappa_J\cup\{t\}$ & $\bar\kappa_I$ & if $\mathbf X$ is even and $f$ is not idempotent \\
\hline
$\emptyset$ & $\bar\kappa_J$ & $\bar\kappa_I\cup\{t\}$ & if $\mathbf X$ is even and $f$ is idempotent\\
\hline
\end{tabular}
\end{center}
\end{table}
\\
for $u\in\kappa$,
\begin{equation}\label{XHiGYkESZUL}
\begin{array}{llll}
L_u&=&\{x\in X : \res{}{x}{x}=u\},\\
H_u&=&\{x\in L_u : {x}{\nega{u}}<x\}=\{x\in L_u : {\mbox{$x$ is $u$-invertible}} \},\footnotemark\\
\accentset{\bullet}H_u&=&\{ \accentset{\bullet}x : x\in H_u\}
\mbox{
where $\accentset{\bullet}x={x}{\nega{u}}$,
}
\end{array}
\end{equation}
\footnotetext{Here and in the next row, $u\in\kappa_I$. We say that $x\in L_u$ is $u$-invertible if there is $y\in L_u$ such that $xy=u$.}
\begin{equation}\label{DEFcsopi}
G_u=\left\{
\begin{array}{ll}
L_u & \mbox{if $u\notin\kappa_I$}\\
L_u\setminus\accentset{\bullet}H_u & \mbox{if $u\in\kappa_I$}\\
\end{array}
\right.
,
\end{equation}
\begin{equation}\label{}
\leq_u\ =\ \leq\,\cap\ (G_u\times G_u)
\end{equation}
 \begin{equation}\label{IgyTorzulaSzorzat}
x\cdot_u y=
\left\{
\begin{array}{ll}
xy & \mbox{if $u\notin\kappa_I$}\\
\res{}{(\res{}{{x}{y}}{u})}{u} & \mbox{if $u\in\kappa_I$}\\
\end{array}
\right.
,
\end{equation}
for $x\in G_u$,
\begin{equation}\label{EzLeSzainVerZ}
x^{-1_u}=\res{}{x}{u},
\end{equation} 
and for $u,v\in\kappa$ such that $u\leq_\kappa v$, $\varsigma_{u\to v} : G_u\to G_v$ is defined by
\begin{equation}\label{MapAzSzorzas}
\varsigma_{u\to v}(x)={v}{x}
.
\end{equation}
\end{enumerate}
\begin{enumerate}[label={(B)}]
\item
Given a bunch of layer groups
$$
\mbox{
$
\mathcal X=\langle \textbf{\textit{G$_u$}},\textbf{\textit{H$_u$}}, \varsigma_{u\to v} \rangle_{{\langle \kappa_o, \kappa_J, \kappa_I, \leq_\kappa\rangle}}
$
with
$
\textbf{\textit{G$_u$}}=(G_u,\preceq_u,\cdot_u,\ { }^{-1_u},u)
\ \ 
(u\in\kappa)
$,
}
$$
$$
\mathbf X_{\mathcal X}=(X,\leq,\cdot,\ite{},t,f)
$$
is an involutive FL$_e$-chain with residual complement $\komp$,
called the {\em involutive FL$_e$-chain of $\mathcal X$}
with
\begin{equation}\label{EZazX}
X=\displaystyle\dot\bigcup_{u\in \kappa}L_u
,
\end{equation}
\indent where (according to Definition~\ref{DEFbunch})
$$
\kappa=\kappa_o\cup\kappa_J\cup\kappa_I,
$$
\indent
for $u\in\kappa_I$,
\begin{equation}\label{Hukeszul}
\mbox{
$\accentset{\bullet}H_u=\{\accentset{\bullet}{h} : h\in H_u\}$ 
}
\end{equation}
\indent
for $u\in\kappa$,
\begin{equation}\label{IkszU}
L_u=\left\{
\begin{array}{ll}
G_u & \mbox{ if $u\not\in\kappa_I$}\\
G_u\,
\overset{.}{\cup}\, \accentset{\bullet}H_u & \mbox{ if $u\in\kappa_I$}\\
\end{array}
\right. ,
\end{equation}
for $u,v\in\kappa$, $u\leq_\kappa v$,
\begin{equation}\label{zetaRANGe}
\zeta_{u\to v} : L_u \to G_v
\end{equation}
extends $\varsigma_{u\to v}$
by
\begin{equation}\label{DEFzeta}
\zeta_{u\to v}(x)=
\left\{
\begin{array}{ll}
\varsigma_{u\to v}(x) & \mbox{ if $x\in G_u$,}\\
\varsigma_{u\to v}(a) & \mbox{ if $x=\accentset{\bullet}a\in \accentset{\bullet}H_u$, $u\in\kappa_I$,}\\
\end{array}
\right.
\end{equation}
for $u,v\in\kappa$, $x\in L_u$ and $y\in L_v$,
$x<y$ is defined by
\begin{eqnarray}
\label{INNENfrom}
\zeta_{u\to uv}(x)  \prec_{uv} \zeta_{v\to uv}(y),\footnotemark
\mbox{ or }
\\
\label{sgdJHJHKJKH3} 
\zeta_{u\to uv}(x)  = \zeta_{v\to uv}(y)
\mbox{ and one of the following three conditions holds }
\\
\label{kisebbAB}
u<_\kappa v
, y\in G_v,
\\
\label{kisebbC}
u=v\in\kappa_I, x\in\accentset{\bullet}H_u, y\in G_u,
\\
\label{kisebbD}
\kappa_I\in u>_\kappa v, x\in\accentset{\bullet}H_u,
\end{eqnarray}
\footnotetext{\ Here and also in (\ref{szorzatJOL}) $uv$ stands for $\max_{\leq_\kappa}(u,v)$. This notation does not cause any inconsistency with the notation of Theorem~\ref{BUNCH_X}/\ref{errefere} 
since for any two positive idempotent elements $u,v$ of an odd or even involutive FL$_e$-chain $(X,\leq,\cdot,\ite{},t,f)$ it holds true that $uv=\max_{\leq}(u,v)$ which is further equal to $\max_{\leq_\kappa}(u,v)$ by (\ref{IGYleszSKELETON}).}
for $u,v\in\kappa$, $x\in L_u$ and $y\in L_v$,
\begin{equation}\label{szorzatJOL}
xy
=
\left\{
\begin{array}{ll}
(\zeta_{u\to uv}(x)\cdot_{uv}\zeta_{v\to uv}(y))^\bullet & 
\mbox{ if $u\neq v$, $uv\in\kappa_I$, $x$ or $y$ is in $\accentset{\bullet} H_{uv}$}\\
\left(\zeta_{u\to uv}(x)\cdot_{uv}\zeta_{v\to uv}(y)\right)^\bullet & 
\mbox{ 
if $u=v\in\kappa_I$,
$\zeta_{u\to uv}(x)\cdot_{uv}\zeta_{v\to uv}(y)\in H_{uv}$,
$\neg(x,y\in H_{uv})$
}\\
\zeta_{u\to uv}(x)\cdot_{uv}\zeta_{v\to uv}(y) & \mbox{ otherwise}\\
\end{array}
\right. 
\end{equation}
for $x,y\in X$,
\begin{equation}\label{IgYaReSi}
\res{}{x}{y}=\nega{({x}{\nega{y}})},
\end{equation}

where
for $x\in X$ the residual complement $\komp$ is defined by
\begin{equation}\label{SplitNega}
\nega{x}
=
\left\{
\begin{array}{ll}
\left(\zeta_{u\to u}(x)^{-1_u}\right)^\bullet	& \mbox{ if $u\in\kappa_I$ and $x\in H_u$}\\
{\zeta_{u\to u}(x)^{-1_u}}_{\downarrow_{G_u}} & \mbox{ if $u\in\kappa_J$ and $x\in G_u$}\\
\zeta_{u\to u}(x)^{-1_u}		& \mbox{ otherwise}\\
\end{array}
\right.
.
\end{equation}
\begin{equation}\label{tLESZez}
\mbox{
$t$ is the least element of $\kappa$,
}
\end{equation}
\begin{equation}\label{tLESZaz}
\mbox{
$f$ is the residual complement of $t$.
}
\end{equation}
In addition, 
$\mathbf X_{\mathcal X}$ is odd if $t\in\kappa_o$, 
even with a non-idempotent falsum if $t\in\kappa_J$, and 
even with an idempotent falsum if $t\in\kappa_I$.

\end{enumerate}

\begin{enumerate}[start=1,label={(C)}]
\item
Given a bunch of layer groups $\mathcal X$ it holds true that 
$\mathcal X_{({\mathbf X}_\mathcal X)}=\mathcal X$, and
given an odd or even involutive FL$_e$-chain $\mathbf X$ it holds true that $\mathbf X_{(\mathcal X_\mathbf X)}\simeq\mathbf X$.
\footnote{\label{ModifiCaTO}
If Definition~\ref{DEFbunch} is slightly modified in such a way that the
$\accentset{\bullet}{\textbf{\textit{H$_u$}}}$'s
are \lq\lq stored\rq\rq\ in the definition of a bunch
(like ${\mathcal X}=\langle \textbf{\textit{G$_u$}},\textbf{\textit{H$_u$}},\accentset{\bullet}{\textbf{\textit{H$_u$}}}, \varsigma_{u\to v} \rangle_{\langle \kappa_o, \kappa_J, \kappa_I, \leq_\kappa\rangle}$),
and 
instead of taking a copy $\accentset{\bullet}{H}_u$ of $H_u$ in (\ref{Hukeszul}), that stored copy is used in the construction of Theorem~\ref{BUNCH_X}/\ref{errefere},
then also
$\mathbf X_{(\mathcal X_\mathbf X)}=\mathbf X$
holds.
Then, Theorem~\ref{BUNCH_X} describes a bijection, in a constructive manner, between the classes $\mathfrak I^{\mathfrak c}_{\mathfrak 0\mathfrak1}$ and $\mathfrak B_{\mathfrak G}$.
For our purposes in the present paper stating only the isomorphism will be sufficient.}
\qed
\end{enumerate}
\end{theorem}
\begin{proof}
There are three differences compared to \cite[Theorem~8.1]{JenRepr2020}. 
\begin{enumerate}[(I)]
\item\label{Itodo}
Concerning (\ref{INNENfrom})--(\ref{kisebbD}), in \cite[Theorem~8.1]{JenRepr2020} the order $\leq$ was given as follows:
for $u,v\in\kappa$, $x\in L_u$ and $y\in L_v$,
\begin{equation}\label{RendeZesINNOVATIVAN}
\mbox{$x<y$ iff $\rho_{uv}(x)<_{uv}\rho_{uv}(y)$ or 
($\rho_{uv}(x)=\rho_{uv}(y)$ and $u<_\kappa v$),}\\
\end{equation}
where
for $v\in\kappa$, $\rho_v : X\to X$ is defined by
\begin{equation*}\label{P52}
\begin{array}{lll}
\rho_v(x)&=&
\left\{
\begin{array}{ll}
\varsigma_{u\to v}(x) & \mbox{ if $x\in G_u$ and $u<_\kappa v$},\\
x & \mbox{ if $x\in G_u$ and $u\geq_\kappa v$},\\
\end{array}
\right. 
\\
\rho_v(\accentset{\bullet}x)&=&
\left\{
\begin{array}{ll}
\varsigma_{u\to v}(x) & \mbox{ if $\accentset{\bullet}x\in \accentset{\bullet}H_u$ and $\kappa_I\ni u<_\kappa v$},\\
\accentset{\bullet}x & \mbox{ if $\accentset{\bullet}x\in \accentset{\bullet}H_u$ and $\kappa_I\ni u\geq_\kappa v$},\\
\end{array}
\right. 
\end{array}
\end{equation*}
and
the ordering $<_u$ of $L_u$ is given by
\begin{equation}\label{KibovitettRendezesITTIS}
\begin{array}{ll}
\mbox{ $\leq_u\,=\,\preceq_u$ if $u\notin\kappa_I$, whereas 
if $u\in\kappa_I$ then 
$\leq_u$ extends $\preceq_u$ 
by letting 
}\\
\mbox{
$\accentset{\bullet} a<_u\accentset{\bullet} b$ and $x<_u\accentset{\bullet} a<_uy$
for $a,b\in H_u$, $x,y\in G_u$ with $a\prec_u b$, $x\prec_u a\preceq_u y$.
}
\end{array}
\end{equation}

\item\label{IItodo}
Concerning (\ref{szorzatJOL}), 
in \cite[Theorem~8.1]{JenRepr2020} $\cdot$ was given by
\begin{equation}\label{EgySzeruTe2}
{x}{y}={\rho_{uv}(x)}\mathbin{\bigcdot_{uv}}{\rho_{uv}(y)},
\end{equation}
where
for $u\in\kappa$, $\gamma_u : L_u \to G_u$
\begin{equation}\label{DEFgamma}
\gamma_u(x)=
\left\{
\begin{array}{ll}
x & \mbox{ if $x\in G_u$}\\
a & \mbox{ if $x=\accentset{\bullet}a\in \accentset{\bullet}H_u$, $u\in\kappa_I$}\\
\end{array}
\right.
,
\end{equation}
and for $x,y\in L_u$,
\begin{equation}\label{uPRODigy0}
{x}\mathbin{\bigcdot_u}{y}=\left\{
\begin{array}{ll}
{\left({\gamma_u(x)}\cdot_u{\gamma_u(y)}\right)}^\bullet	& \mbox{ if $u\in\kappa_I$, ${\gamma_u(x)}\cdot_u{\gamma_u(y)}\in H_u$ and $\neg(x,y\in H_u)$}\\
{\gamma_u(x)}\cdot_u{\gamma_u(y)}		& \mbox{ if $u\in\kappa_I$, ${\gamma_u(x)}\cdot_u{\gamma_u(y)}\notin H_u$ or $x,y\in H_u$}\\
x\cdot_u y& \mbox{ if $u\notin\kappa_I$}\\
\end{array}
\right. 
.
\end{equation}

\item\label{IIItodo}
Concerning (\ref{SplitNega}), in \cite[Theorem~8.1]{JenRepr2020}
$\komp$ was defined by 
\begin{equation}\label{regiKOMP}
\nega{x}
=
\left\{
\begin{array}{ll}
a^{-1_u}		& \mbox{ if $u\in\kappa_I$ and $x=\accentset{\bullet}a\in \accentset{\bullet}H_u$}\\
\left(x^{-1_u}\right)^\bullet	& \mbox{ if $u\in\kappa_I$ and $x\in H_u$}\\
x^{-1_u}	& \mbox{ if $u\in\kappa_I$ and $x\in G_u\setminus H_u$}\\
{x^{-1_u}}_{\downarrow_{G_u}} & \mbox{ if $u\in\kappa_J$ and $x\in G_u$}\\
x^{-1_u}		& \mbox{ if $u\in\kappa_o$ and $x\in G_u$}\\
\end{array}
\right. 
.
\end{equation} 
\end{enumerate}

\ref{Itodo}
First notice that for $u\in\kappa_I$ the definition of the ordering in (\ref{KibovitettRendezesITTIS}) can equivalently be given by any of the following ones
\begin{equation}\label{KATEGOR_KibovitettRendezesITTIS55}
\begin{array}{ll}
\mbox{ 
$x\leq_u y$
iff
$\gamma_u(x)\prec_u \gamma_u(y)$
or 
$\left(
\mbox{
$\gamma_u(x)=\gamma_u(y)$
and 
$(x\in\accentset{\bullet}H_u$ or $y\in G_u)$
}
\right)
$
}
\end{array}
\end{equation}
\begin{equation}\label{KATEGOR_KibovitettRendezesITTIS66}
\begin{array}{ll}
\mbox{ 
$x<_u y$
iff 
$\gamma_u(x)\prec_u \gamma_u(y)$
or 
$\left(
\mbox{
$\gamma_u(x)=\gamma_u(y)$, 
$x\in\accentset{\bullet}H_u$,
$y\in G_u$
}
\right)
$
}
\end{array}
\end{equation}
since for $u\in\kappa_I$, the meaning of the definition of the ordering in (\ref{KibovitettRendezesITTIS}) is that for any element $a$ in a subgroup, its dotted copy $\accentset{\bullet}a$ is inserted just below $a$, and any of (\ref{KATEGOR_KibovitettRendezesITTIS55}) and (\ref{KATEGOR_KibovitettRendezesITTIS66}) expresses the same.
Also notice that for $v\in\kappa$, $\rho_v : X\to X$ can be written in the following simpler form:
\begin{equation}\label{P5}
\rho_v(x)
=
\left\{
\begin{array}{ll}
x & \mbox{ if $x\in L_u$ and $u\geq_\kappa v$}\\
\varsigma_{u\to v}(\gamma_u(x)) & \mbox{ if $x\in L_u$ and $u<_\kappa v$}\\
\end{array}
\right. 
,
\end{equation}
and that for $u,v\in\kappa$, $u\leq_\kappa v$, 
\begin{equation}\label{KompoZicio}
\zeta_{u\to v}=\varsigma_{u\to v}\circ\gamma_u
\end{equation}
Therefore, by (\ref{P5}), (\ref{RendeZesINNOVATIVAN}) is equivalent to
$$
\left\{
\begin{array}{ll}
\zeta_{u\to v}(x)  \leq_v y
&
\mbox{if $u<_\kappa v$} \\
x <_u y
&
\mbox{if $u=v$} \\
x <_u \zeta_{v\to u}(y)
&
\mbox{if $u>_\kappa v$} \\
\end{array}
\right.
$$
and by (\ref{KATEGOR_KibovitettRendezesITTIS55}), (\ref{KibovitettRendezesITTIS}),
and (\ref{KATEGOR_KibovitettRendezesITTIS66}), respectively,
it is further equivalent to
$$
\footnotesize\left\{
\begin{array}{ll}
\zeta_{u\to v}(x)  \preceq_v y
&
\mbox{if $u<_\kappa v\notin\kappa_I$} \\
\gamma_v(\zeta_{u\to v}(x))\prec_v \gamma_v(y)
\mbox{ or } 
\left(
\gamma_v(\zeta_{u\to v}(x))=\gamma_v(y)
\mbox{ and } 
\left(
\zeta_{u\to v}(x)\in\accentset{\bullet}H_v
\mbox{ or } 
y\in G_v
\right)
\right)
&
\mbox{if $u<_\kappa v\in\kappa_I$} \\
x\prec_u y
&
\mbox{if $u=v\notin\kappa_I$} \\
\gamma_u(x)\prec_u \gamma_u(y)
\mbox{ or }
\left(
\gamma_u(x)=\gamma_u(y), 
x\in\accentset{\bullet}H_u,
y\in G_u
\right)
&
\mbox{if $u=v\in\kappa_I$} \\
x \prec_u \zeta_{v\to u}(y)
&
\mbox{if $\kappa_I\not\in u>_\kappa v$} \\
\gamma_u(x)\prec_u \gamma_u(\zeta_{v\to u}(y))
\mbox{ or }
\left(
\gamma_u(x)=\gamma_u(\zeta_{v\to u}(y)), 
x\in\accentset{\bullet}H_u,
\zeta_{v\to u}(y)\in G_u
\right)
&
\mbox{if $\kappa_I\in u>_\kappa v$} \\
\end{array}
\right.
.
$$
Now, $\zeta_{v\to u}(y)\in G_u$ always holds in the last row, see (\ref{DEFzeta}), 
and
$\zeta_{u\to v}(x)\in\accentset{\bullet}H_v$ cannot hold in the second row
since
$G_v$ and $\accentset{\bullet}H_v$ are disjoint by (\ref{IkszU}).
In addition, $\gamma_v\circ\zeta_{u\to v}=\zeta_{u\to v}$ holds by (\ref{DEFzeta}).
Because of these, the condition above is equivalent to
$$
\footnotesize
\left\{
\begin{array}{ll}
\zeta_{u\to v}(x)  \preceq_v y
&
\mbox{if $u<_\kappa v\notin\kappa_I$} \\
\zeta_{u\to v}(x)\prec_v \gamma_v(y)
\mbox{ or } 
\left(
\zeta_{u\to v}(x)=\gamma_v(y)
\mbox{ and } 
y\in G_v
\right)
&
\mbox{if $u<_\kappa v\in\kappa_I$} \\
x\prec_u y
&
\mbox{if $u=v\notin\kappa_I$} \\
\gamma_u(x)\prec_u \gamma_u(y)
\mbox{ or }
\left(
\gamma_u(x)=\gamma_u(y), 
x\in\accentset{\bullet}H_u,
y\in G_u
\right)
&
\mbox{if $u=v\in\kappa_I$} \\
x \prec_u \zeta_{v\to u}(y)
&
\mbox{if $\kappa_I\not\in u>_\kappa v$} \\
\gamma_u(x)\prec_u \zeta_{v\to u}(y)
\mbox{ or }
\left(
\gamma_u(x)=\zeta_{v\to u}(y), 
x\in\accentset{\bullet}H_u
\right)
&
\mbox{if $\kappa_I\in u>_\kappa v$} \\
\end{array}
\right.
.
$$
Keeping in mind that by (\ref{IkszU}), $L_v=G_v$ if $v\notin\kappa_I$, and
referring to (\ref{DEFgamma}), \ref{(G1)}, and (\ref{DEFzeta}), finally we obtain that the above condition is equivalent to (\ref{INNENfrom})--(\ref{kisebbD}).

\bigskip
\ref{IItodo}
By (\ref{DEFgamma}), 
$\mathbin{\bigcdot_u}$ in (\ref{uPRODigy0}) can be written as
\begin{equation}\label{uPRODigy}
{x}\mathbin{\bigcdot_u}{y}=
\left\{
\begin{array}{ll}
\left({\gamma_u(x)}\cdot_u{\gamma_u(y)}\right)^\bullet
& \mbox{ if 
$u\in\kappa_I$, 
${\gamma_u(x)}\cdot_u{\gamma_u(y)}\in H_u$,
$\neg(x,y\in H_u)$
}\\
{\gamma_u(x)}\cdot_u{\gamma_u(y)}		& \mbox{ otherwise}\\
\end{array}
\right.
.
\end{equation}
By (\ref{DEFzeta}) and (\ref{DEFgamma}),
for $u,v\in\kappa$, $u\leq_\kappa v$,
\begin{equation}\label{zetaMI}
\zeta_{u\to v}
=
\varsigma_{u\to v}\circ\gamma_u,
\end{equation}
and hence by \ref{(G3)},
\begin{equation}\label{zetaMAPSto}
\mbox{
for $u<_\kappa v\in\kappa_I$, $\zeta_{u\to v}$ maps into $H_v$.
}
\end{equation}
If 
$u\geq_\kappa v$
then
{\footnotesize
\begin{eqnarray*} 
xy
&\overset{(\ref{EgySzeruTe2})}{=}&
{\rho_{uv}(x)}\mathbin{\bigcdot_{uv}}{\rho_{uv}(y)}
\\
&\overset{(\ref{P5}), (\ref{zetaMI})}{=}&
\left\{
\begin{array}{ll}
\gteM{u}{x}{\zeta_{v\to u}(y)} & \mbox{ if $uv=u>_\kappa v$}\\
\gteM{u}{x}{y} & \mbox{ if $uv=u= v$}\\
\end{array}
\right. 
\\
&\overset{(\ref{IkszU})}{=}&
\left\{
\begin{array}{ll}
\gteM{u}{x}{\zeta_{v\to u}(y)} & \mbox{ if $\kappa_I\not\ni uv=u>_\kappa v$}\\
\gteM{u}{x}{\zeta_{v\to u}(y)} & \mbox{ if $\kappa_I\ni uv=u>_\kappa v$, $x\in H_u$}\\
\gteM{u}{x}{\zeta_{v\to u}(y)} & \mbox{ if $\kappa_I\ni uv=u>_\kappa v$, $x\in G_u\setminus H_u$}\\
\gteM{u}{x}{\zeta_{v\to u}(y)} & \mbox{ if $\kappa_I\ni uv=u>_\kappa v$, $x\in \accentset{\bullet}H_u$}\\
\gteM{u}{x}{y} & \mbox{ if $u= v$}\\
\end{array}
\right. 
\\
&\overset{(\ref{zetaMAPSto}),(\ref{DEFgamma}),(\ref{uPRODigy})}{=}&
\left\{
\begin{array}{ll}
\gamma_u(x)\cdot_u\gamma_u(\zeta_{v\to u}(y)) & \mbox{ if $\kappa_I\not\ni uv=u>_\kappa v$}\\
\gamma_u(x)\cdot_u\gamma_u(\zeta_{v\to u}(y)) & \mbox{ if $\kappa_I\ni uv=u>_\kappa v$, $x\in H_u$}\\
\gamma_u(x)\cdot_u\gamma_u(\zeta_{v\to u}(y))\footnotemark & \mbox{ if $\kappa_I\ni uv=u>_\kappa v$, $x\in G_u\setminus H_u$}\\
\left(\gamma_u(x)\cdot_{}\gamma_u(\zeta_{v\to u}(y))\right)^\bullet & \mbox{ if $\kappa_I\ni uv=u>_\kappa v$, $x\in \accentset{\bullet}H_u$}\\
\left({\gamma_u(x)}\cdot_u\gamma_v(y)\right)^\bullet & 
\mbox{ 
if $\kappa_I\ni uv=u=v$,
${\gamma_u(x)}\cdot_u{\gamma_{v}(y)}\in H_u$,
$\neg(x,y\in H_u)$
}\\
\gamma_u(x)\cdot_u\gamma_v(y) & \mbox{ otherwise}\\
\end{array}
\right. 
\\
&\overset{\ref{(G1)},(\ref{zetaMI})}{=}&
\left\{
\begin{array}{ll}
\zeta_{u\to uv}(x)\cdot_{uv}\zeta_{v\to uv}(y) & \mbox{ if $\kappa_I\not\ni uv=u>_\kappa v$}\\
\zeta_{u\to uv}(x)\cdot_{uv}\zeta_{v\to uv}(y) & \mbox{ if $\kappa_I\ni uv=u>_\kappa v$, $x\in G_u$}\\
\left(\zeta_{u\to uv}(x)\cdot_{uv}\zeta_{v\to uv}(y)\right)^\bullet & \mbox{ if $\kappa_I\ni u>_\kappa v$, $x\in \accentset{\bullet}H_u$}\\
\left(\zeta_{u\to uv}(x)\cdot_{uv}{\zeta_{v\to uv}(y)}\right)^\bullet & 
\mbox{ 
if $\kappa_I\ni u=v$,
$\zeta_{u\to uv}(x)\cdot_{uv}{\zeta_{v\to uv}(y)}\in H_u$,
$\neg(x,y\in H_u)$
}\\
\zeta_{u\to uv}(x)\cdot_{uv}\zeta_{v\to uv}(y) & \mbox{ otherwise}\\
\end{array}
\right. 
\\
&\overset{(\ref{zetaMI})}{=}&
\left\{
\begin{array}{ll}
(\zeta_{u\to uv}(x)\cdot_{uv}\zeta_{v\to uv}(y))^\bullet & 
\mbox{ if $\kappa_I\ni u>_\kappa v$, $x\in\accentset{\bullet} H_{uv}$}\\
\left(\zeta_{u\to uv}(x)\cdot_{uv}\zeta_{v\to uv}(y)\right)^\bullet & 
\mbox{ 
if $\kappa_I\ni u=v$,
$\zeta_{u\to uv}(x)\cdot_{uv}\zeta_{v\to uv}(y)\in H_{uv}$,
$\neg(x,y\in H_{uv})$
}\\
\zeta_{u\to uv}(x)\cdot_{uv}\zeta_{v\to uv}(y) & \mbox{ otherwise}\\
\end{array}
\right.
\end{eqnarray*} 
}
\footnotetext{Here we use that for groups $H_u\leq G_u$,
$(G_u\setminus H_u)H_u\cap H_u=\emptyset$ holds.}
therefore, 
for
$u,v\in\kappa$, $x\in L_u$ and $y\in L_v$,
(\ref{szorzatJOL}) holds.

\bigskip
\ref{IIItodo}
By (\ref{DEFgamma}), (\ref{DEFzeta}), and \ref{(G1)}, (\ref{regiKOMP}) is clearly equivalent to (\ref{SplitNega}).
\end{proof}

Corollary~\ref{HogyanLatszikCOR} shows how \lq $\mathbf X$ is embeddable into $\mathbf Y$\rq\, 
can be seen from the respective group representations  $\mathcal X$ and $\pazocal Y$.
Since the universe of the skeleton and the universe of the layer groups are subsets of the universe of the algebra, see Theorem~\ref{BUNCH_X}/\ref{errefere}, any mapping $\iota$ from $\mathbf X$ to $\mathbf Y$ induces (by simple restriction of $\iota$ to that subset) a mapping $\iota_\kappa$ from the skeleton of $\mathbf X$, 
and in every layer $u$, a mapping $\iota_u$ from the $u^{\rm th}$-layer group \textbf{\textit{G$\iksz _u$}} of $\mathbf X$.

The following statement is an easy consequence of \cite[Theorem~3.6]{JScategorical} (see, e.g., \cite[Lemma~2.10/(4)]{JSamalg}).
\begin{corollary}
\label{HogyanLatszikCOR}
Let 
$$
\begin{array}{l}
\mathbf X=(X,\leq_X,\teiksz,\ite{}\iksz ,t\iksz ,f\iksz ),
\\
\mathbf Y=(Y,\leq_Y,\teipsz,\ite{}\ipsz ,t\ipsz ,f\ipsz )
\end{array}
$$ 
be algebras both either in $\mathfrak I^\mathfrak c_0$ or in $\mathfrak I^\mathfrak c_1$, and their respective group representations be
\begin{equation}\label{EQrend875skjdhJG}
\begin{array}{l}
{\mathcal X}=\langle \textbf{\textit{G$\iksz _u$}},\textbf{\textit{H$\iksz _u$}}, \varsigma_{u\to v}\iksz   \rangle_{\langle \kappa_o\iksz , \kappa_J\iksz , \kappa_I\iksz ,\leq_{\kappa^{\scaleto{(\mathbf X)}{3pt}}}\rangle}
\\
{\pazocal Y}=\langle \textbf{\textit{G$\ipsz _u$}},\textbf{\textit{H$\ipsz _u$}}, \varsigma_{u\to v}\ipsz   \rangle_{\langle \kappa_o\ipsz , \kappa_J\ipsz , \kappa_I\ipsz ,\leq_{\kappa^{\scaleto{(\mathbf Y)}{3pt}}}\rangle}
\end{array}
\end{equation}
Then
$\iota : \mathbf X\to\mathbf Y$ is an embedding
if and only if the following conditions hold.
\begin{enumerate}[start=1,label={(E\arabic*)}]
\item \label{E1}
$\iota_\kappa:=\iota_{|_{\kappa^{\scaleto{(\mathbf X)}{3pt}}}}$ is an $o$-embedding 
of the skeleton 
of $\mathcal X$ into the skeleton 
of $\pazocal Y$, which
preserves the least element and the partition,
\item\label{E2LesZEz}
For every $u\in\kappa\iksz $, 
$\iota_u:=\iota_{|_{\textbf{\textit{G$^{\scaleto{(\mathbf X)}{3pt}}_u$}}}}$ is an ($o$-group) embedding of the $u^{\rm th}$-layer group \textbf{\textit{G$\iksz _u$}} of $\mathcal X$ into the $\iota_\kappa(u)^{\rm th}$-layer group \textbf{\textit{G$\ipsz _{\iota_\kappa(u)}$}} of $\pazocal Y$
such that 
\begin{enumerate}
\item
for every $u,v\in\kappa\iksz$, $u\leq v$ the following diagram 
commutes,
\begin{figure}[ht]
\begin{diagram}
\textbf{\textit{G$\iksz _u$}} & \rEmbed_{\iota_u} & \textbf{\textit{G$\ipsz _{\iota_\kappa(u)}$}} \\
\dTo^{\varsigma_{u\to v}\iksz } & & \dTo_{\varsigma_{{\iota_\kappa(u)}\to {\iota_\kappa(v)}}\ipsz } & \\
\textbf{\textit{G$\iksz _v$}} & \rEmbed_{\iota_v} & \textbf{\textit{G$\ipsz _{\iota_\kappa(v)}$}} \\
\end{diagram}
\label{HOMO_NocsaKK}
\end{figure}
\item
if $u\in\kappa_I\iksz $ then $\iota_u$ maps
$H\iksz _u$ to $H\ipsz _{\iota_\kappa(u)}$
and
$G\iksz _u\setminus H\iksz _u$ to $G\ipsz _{\iota_\kappa(u)}\setminus H\ipsz _{\iota_\kappa(u)}$
, 
\item \label{dsghjKJ87G}
if $u\in\kappa_J\iksz $ then $\iota_u$ preserves the cover of the unit element $u$ of \textbf{\textit{G$\iksz _u$}}. 
\end{enumerate}
\end{enumerate}
\end{corollary}

\section{Densification in $\mathfrak I^{\mathfrak{sl}}_{\mathfrak 0}$ and $\mathfrak I^{\mathfrak{sl}}_{\mathfrak 0,\mathfrak{symm}}$}

An FL$_e$-chain $\mathbf Y$ fills the gap $x_1,x_2$ of the FL$_e$-chain $\mathbf X$ if there exists an embedding $\iota: X\to Y$ and $y\in Y$ such that $\iota(x_1)<y<\iota(x_2)$.
A nontrivial variety $\mathtt V$ is said to be {\em densifiable} 
if every gap of every chain in $\mathtt V$ can be filled by another chain in $\mathtt V$. 
For a variety, densifiability is sufficient for densification, as shown in
\begin{proposition}\cite[Proposition 2.2]{BaldiTerui}\label{BaldiTeruiREF} 
Let $L$ be a language of algebras and $\mathtt V$ a densifiable variety of type $L$. Then every chain $\mathbf X$ of cardinality $\delta>1$ is embeddable into a dense chain of cardinality $\delta+\aleph_0+|L|$.
\end{proposition}
Therefore, using Lemma~\ref{kjh986}, first we prove Theorem~\ref{EzaTuTtIDENSIFIABLE}.
We say that the bunch $\mathcal Y$ is an extension of a bunch $\mathcal X$ if 
$\kappa\iksz\subseteq\kappa\ipsz$,
$t\iksz=t\ipsz$,
$\textbf{\textit{G$\iksz_u$}}=\textbf{\textit{G$\ipsz_u$}}$
and
$\textbf{\textit{H$\iksz_u$}}=\textbf{\textit{H$\ipsz_u$}}$
($u\in\kappa\iksz_I$),
and
$\varsigma\iksz_{u\to v}=\varsigma\ipsz_{u\to v}$ ($u,v\in\kappa\iksz$, $u\leq_{\kappa\iksz} v$).
Obviously, if the bunch $\mathcal Y$ is an extension of a bunch $\mathcal X$ then $\mathcal X$ embeds into $\mathcal Y$ via the identity mapping in each coordinate, cf.\ Corollary~\ref{HogyanLatszikCOR},
and consequently, the FL$_e$-chain corresponding to $\mathcal X$ embeds into the FL$_e$-chain corresponding to $\mathcal Y$. 
The following lemma shows a particular instance of it, namely, a way of inserting a single new element into the skeleton of any bunch and extending the original bunch to the one-element-larger skeleton.

\begin{definition}
For a bunch of layer groups
$
\mathcal X=\langle \textbf{\textit{G$_u$}},\textbf{\textit{H$_u$}}, \varsigma_{u\to v} \rangle_{{\langle \kappa_o, \kappa_J, \kappa_I, \leq_\kappa\rangle}}
$
and $v\in\kappa\setminus\kappa_J$,
define $\mathcal X_{ v^\star}$ as follows.
\begin{enumerate}[(V1)]
\item\label{V1} 
Insert a new element $v^\star$ into the skeleton just above $v$, that is, set $a<_\kappa v^\star$ for $a\leq_\kappa v$ and $a>_\kappa v^\star$ for $a>_\kappa v$.
\item
Let 
\textbf{\textit{G$_{v^\star}$}} be an isomorphic copy of \textbf{\textit{G$_v$}} 
with isomorphism 
\begin{equation}\label{izo1}
\varphi_{v\to v_\star} : G_v \to G_{v^\star},
\end{equation}
and 
\item\label{V3} 
extend the system of transitions by letting, 
for $a\leq_\kappa v$, $\varsigma_{a\to v^\star}=\varphi_{v\to v_\star}\circ\varsigma_{a\to v}$, 
and 
for $a\geq_\kappa v^\star$, $\varsigma_{v^\star\to a}=\varsigma_{v\to a}\circ\varphi_{v\to v_\star}^{-1}$.
\item
Set $v^\star\in\kappa_I$ and let \textbf{\textit{H$_{v^\star}$}}=\textbf{\textit{G$_{v^\star}$}}.
\item\label{V5}
For $y\in G_v$ let $G_{v^\star}\ni y^\star=\varphi_{v\to v_\star}(y)$.
\end{enumerate}
For a bunch of layer groups
$
\mathcal X=\langle \textbf{\text{G$_u$}},\textbf{\textit{H$_u$}}, \varsigma_{u\to v} \rangle_{{\langle \kappa_o, \kappa_J, \kappa_I, \leq_\kappa\rangle}}
$
and $v\in\kappa\setminus\{t\}$,
define $\mathcal X_{v_\star}$ as follows.
\begin{enumerate}[(v1)]
\item
Insert a new element $v_\star$ into the skeleton just below $v$, that is, set $a<_\kappa v_\star$ for $a<_\kappa v$ and $a>_\kappa v_\star$ for $a\geq_\kappa v$.
\item\label{vS2} 
Let 
\textbf{\textit{G$_{v_\star}$}} be an isomorphic copy of \textbf{\textit{G$_v$}} 
with isomorphism $\varphi_{v\to v_\star} : G_v \to G_{v_\star}$, and 
\item\label{vS3} 
extend the system of transitions by letting, 
for $a\geq_\kappa v$, $\varsigma_{v_\star\to a}=\varsigma_{v\to a}\circ\varphi_{v\to v_\star}^{-1}$, 
and 
for $a\leq_\kappa v_\star$, $\varsigma_{a\to v_\star}=\varphi_{v\to v_\star}\circ\varsigma_{a\to v}$.
\item 
Set $v_\star\in\kappa_I$ and let \textbf{\textit{H$_{v_\star}$}}=\textbf{\textit{G$_{v_\star}$}}.
\item
For $y\in G_v$ let $G_{v_\star}\ni y_\star=\varphi_{v\to v_\star}(y)$.
\end{enumerate}
\end{definition}

\begin{lemma}\label{kjh986} 
Let
$
\mathcal X=\langle \textbf{\textit{G$_u$}},\textbf{\textit{H$_u$}}, \varsigma_{u\to v} \rangle_{{\langle \kappa_o, \kappa_J, \kappa_I, \leq_\kappa\rangle}}
$
be a bunch of layer groups.
For any $v\in\kappa\setminus\kappa_J$ (resp. $v\in\kappa\setminus\{t\}$),
$\mathcal X_{ v^\star}$ (resp. $\mathcal X_{v_\star}$) is a bunch of layer groups of the same type as $\mathcal X$, 
in which $\mathcal X$ embeds. 
Denoting the 
odd involutive FL$_e$-chain which corresponds to $\mathcal X_{ v^\star}$ (resp. $\mathcal X_{v_\star}$)
by 
$\mathbf X_{v^\star}$ (resp. $\mathbf X_{v_\star}$)
it holds true that 
for any $y\in G_v$, $y^\star$ (resp. $y_\star$) is the upper (resp. lower) cover of $y$ in the ordering of $\mathbf X_{v^\star}$ (resp. $\mathbf X_{v_\star}$).
\end{lemma}
\begin{proof}
The obtained structures are bunches of layer groups.
Indeed, the $\varsigma$'s over $\kappa\cup\{v^\star\}$ (over $\kappa\cup\{v_\star\}$, respectively) are $o$-homomorphisms,
\ref{IDes} and \ref{Kompooot} clearly hold for them, $\kappa_o$, $\kappa_I\cup\{v^\star\}$ (resp. $\kappa_I\cup\{v_\star\}$), and $\kappa_J$ partition $\kappa\cup\{v^\star\}$ (resp. $\kappa\cup\{v_\star\}$), which is totally ordered and its least element is $t$, and the $\varsigma$'s obviously satisfy \ref{(G1)} and \ref{(G3)}.
Next, 
for $a<_\kappa v$, 
$\varsigma_{a\to v^\star}$ (reps. $\varsigma_{a\to v_\star}$)
satisfies \ref{(G2)},
since
$\varsigma_{a\to v}$ satisfies \ref{(G2)}
and
$\varsigma_{a\to v^\star}=\varphi\circ\varsigma_{a\to v}$
(reps. $\varsigma_{a\to v_\star}=\varphi\circ\varsigma_{a\to v}$),
and
for $a\geq_\kappa v^\star $, 
$\varsigma_{v^\star\to a}$ (reps. $\varsigma_{v_\star\to a}$)
satisfies \ref{(G2)}, too,
since $v^\star,v_\star\notin\kappa_J$.
$\mathcal X_{ v^\star}$ and $\mathcal X_{v_\star}$ are clearly extensions of $\mathcal X$.
Therefore, $\mathcal X$ embeds into $\mathcal X_{v_\star}$. 

For $y\in G_v$, by (\ref{kisebbAB}), $y_\star
<y$  follows from 
$v_\star<_\kappa v$
and
$
\zeta_{v_\star\to vv_\star}(y_\star)
=
\zeta_{v_\star\to vv_\star}(\varphi_{v\to v_\star}(y))
\overset{(\ref{DEFzeta})}{=}
\varsigma_{v_\star\to vv_\star}(\varphi_{v\to v_\star}(y))
\overset{\ref{vS3}}{=}
\varsigma_{v\to v}(\varphi_{v\to v_\star}^{-1}(\varphi_{v\to v_\star}(y)))
\overset{\ref{IDes}}{=}
y
\overset{\ref{IDes}}{=}
\zeta_{v\to vv_\star}(y)
$.
We claim that
$y_\star<y$ is a gap in $\mathbf X_{v_\star}$.
Indeed, assume $x\notin\{y,y_\star\}$.
If $x<y_\star$ (resp. $x>y_\star$) is witnessed by (\ref{INNENfrom}) then 
$x<y$ (resp. $x>y$) is witnessed by (\ref{INNENfrom}), too, and vice versa, as it is easy to see using 
$\zeta_{v_\star\to vv_\star}(y_\star)=\zeta_{v\to vv_\star}(y)$.
If $x<y_\star$ (resp. $x>y_\star$) is witnessed by (\ref{sgdJHJHKJKH3}) 
and one of (\ref{kisebbAB})--(\ref{kisebbD}) 
then 
$x<y$ (resp. $x>y$) is witnessed by (\ref{sgdJHJHKJKH3})
and the same line from (\ref{kisebbAB})--(\ref{kisebbD}), too, and vice versa, as it is easy to see using that $v_\star<v$ is a gap in the skeleton of $\mathbf X_{ v^\star}$.
\end{proof}

\begin{theorem}\label{EzaTuTtIDENSIFIABLE}
No subclass of
$\mathfrak I_{\mathfrak 0\mathfrak 1}$
containig 
an algebra from
$\mathfrak I^{\mathfrak{c}}_{\mathfrak 1}$ is densifiable. 
The varieties $\mathfrak I^{\mathfrak{sl}}_{\mathfrak 0}$ and $\mathfrak I^{\mathfrak{sl}}_{\mathfrak 0,\mathfrak{symm}}$ are densifiable.
\end{theorem}
\begin{proof}
An algebra $\mathbf X$ from $\mathfrak I^{\mathfrak{c}}_{\mathfrak 1}$ 
cannot be embedded into a dense algebra in $\mathfrak I_{\mathfrak 0\mathfrak 1}$
since the images of the constants of $\mathbf X$ must be different due to injectivity, 
and hence such an embedding could only be done into an element $\mathbf Y$ in $\mathfrak I_{\mathfrak 1}$, but then, by definition, there would be a gap between the two constants of $\mathbf Y$.

Let $\mathbf X=(X,\leq,\cdot,\ite{},t,f)$ be an odd involutive FL$_e$-chain, 
$\mathcal X$ its corresponding bunch, and let $x< y$ be a gap in $X$. 
Then $\mathbf X$ is nontrivial,
by (\ref{EZazX}), $x\in L_u$ and $y\in L_v$ for some $u,v\in\kappa$, and the condition described in 
(\ref{INNENfrom})--(\ref{kisebbD}) holds.

\begin{enumerate}
\item 
Assume that the condition in (\ref{INNENfrom}) holds, that is, 
\begin{equation}\label{OurCondi1}
\zeta_{u\to uv}(x)  \prec_{uv} \zeta_{v\to uv}(y).
\end{equation}
Then $u\geq_{\kappa} v$ holds.
Indeed, from 
$u<_{\kappa} v$,
$
x
<
\zeta_{u\to uv}(x)
$
would follow by (\ref{kisebbAB})
since
$
G_{uv}
\overset{(\ref{zetaRANGe})}{\ni} 
\zeta_{u\to uv}(x)
=
\zeta_{uv\to uv}(\zeta_{u\to uv}(x))
$,
and hence
$
x
<
\zeta_{u\to uv}(x)
\overset{(\ref{OurCondi1}),\ref{(G1)}}{<}
\zeta_{v\to uv}(y)
\overset{(\ref{zetaMI}),\ref{(G1)}}{=}
\gamma_v(y)
$
would hold.
Now, by (\ref{DEFgamma}),
either $\gamma_v(y)=y$
which ensures
$
x
<
\zeta_{u\to v}(x)
<
y
$,
a contradiction to $x< y$ being a gap in $X$,
or
$\gamma_v(y)$ is the cover of $y\in\accentset{\bullet}H_v$ in the order $\leq_v$ of $L_v$,
and thus
$
G_v
\overset{}{\ni}
\zeta_{u\to v}(x)
\overset{(\ref{KibovitettRendezesITTIS})}{<_v}
\gamma_v(y)
$
implies 
$
\zeta_{u\to v}(x)
<_v
y
$,
and  in turn,
$
x
<
\zeta_{u\to v}(x)
\overset{(\ref{KibovitettRendezesITTIS})}{<}
y
$,
also contradicting to $x< y$ being a gap in $X$.

\begin{enumerate}
\item \label{EgyA} 
If $u=v=t$ then since $X$ is odd, $v\in\kappa_o$ and in turn, $x\in G_t$ follows,
along with 
$v\notin\kappa_J$
since $\kappa_o$ and $\kappa_J$ are disjoint.
By Lemma~\ref{kjh986}, $x^\star$ in $\mathbf X_{ v^\star}$ is in between $x$ and $y$ in the ordering of $\mathbf X_{ v^\star}$.


\item\label{standardDUMA} 
If $u\geq_\kappa v>_\kappa t$ and $y\in G_v$ then by  Lemma~\ref{kjh986},
$y_\star$ in $\mathbf X_{ v_\star}$ is in between $x$ and $y$ in the ordering of $\mathbf X_{ v_\star}$. 

\item\label{EgyC} 
If $u\geq_\kappa v>_\kappa t$ and $y\in \accentset{\bullet}H_v$ then 
$({\zeta_{v\to v}(y)}^\star)^\bullet$ in $\mathbf X_{ v^\star}$ is in between $x$ and $y$ in the ordering of $\mathbf X_{ v^\star}$. 
Indeed, 
$x<({\zeta_{v\to v}(y)}^\star)^\bullet$ is shown by (\ref{INNENfrom}) since
$$
\begin{array}{ccl}
\zeta_{u\to uv^\star}(x)  
& \overset{\ref{V1}}{=} &
\zeta_{u\to u}(x) \\
&=&
\zeta_{u\to uv}(x)  \\
&\overset{(\ref{OurCondi1})}{\prec_{uv}}&
\zeta_{v\to uv}(y)  \\
&=&
\zeta_{v\to u}(y)   \\
&\overset{(\ref{DEFzeta})}{=}&
\varsigma_{v\to u}(
\zeta_{v\to v}(y)
)   \\
&=&
\varsigma_{v\to u}(\varphi_{v\to v_\star}^{-1}
(
\varphi_{v\to v_\star}(\zeta_{v\to v}(y))
))   \\
&\overset{\ref{V3}}{=}&
\varsigma_{v^\star\to u}(\varphi_{v\to v_\star}(\zeta_{v\to v}(y)))   \\
&\overset{(\ref{izo1}),(\ref{DEFzeta})}{=}&
\zeta_{v^\star\to u}((\varphi_{v\to v_\star}(\zeta_{v\to v}(y)))^\bullet)   \\
&\overset{\ref{V5}}{=}&
\zeta_{v^\star\to u}(({\zeta_{v\to v}(y)}^\star)^\bullet)  \\
&\overset{\ref{V1}}{=}&
\zeta_{v^\star\to uv^\star}(({\zeta_{v\to v}(y)}^\star)^\bullet)
,
\end{array}
$$
and
$({\zeta_{v\to v}(y)}^\star)^\bullet<y$ is shown by (\ref{kisebbD})
since
$$
\begin{array}{ccl}
\zeta_{v^\star\to vv^\star}(({\zeta_{v\to v}(y)}^\star)^\bullet)
&\overset{(\ref{DEFzeta})}{=}&
\varsigma_{v^\star\to vv^\star}({\zeta_{v\to v}(y)}^\star)  \\
&\overset{\ref{IDes}}{=} &
{\zeta_{v\to v}(y)}^\star  \\
&\overset{\ref{V5}}{=}&
\varphi_{v\to v_\star}(\zeta_{v\to v}(y))  \\
&\overset{\ref{V3}}{=}&
\varsigma_{v\to v^\star}(\zeta_{v\to v}(y))\\
&\overset{(\ref{KompoZicio})}{=}&
\zeta_{v\to v^\star}(y)\\
&=&
\zeta_{v\to vv^\star}(y)
.
\end{array}
$$

\vskip1cm

\end{enumerate}

\item Assume in the following three items 
(\ref{kisebbAB}), (\ref{kisebbC}), and (\ref{kisebbD}), respectively:
let
\begin{equation}\label{OurCondi2}
\zeta_{u\to uv}(x)=\zeta_{v\to uv}(y).
\end{equation}

\begin{enumerate}
\item \label{KettoA} 
If $u<_\kappa v$ and $y\in G_v$ then $v\neq t$, thus we can consider $\mathcal X_{v_\star}$ 
and conclude the proof as in item (\ref{standardDUMA}) above. 

\item\label{KettoB} 
If $u=v\in\kappa_I$, $x\in \accentset{\bullet}H_u$ and $y\in G_v$ then since $\mathbf X$ is odd, it follows that $v\neq t$. Hence we can consider $\mathcal X_{v_\star}$, and $x<y_\star<y$ follows as above.

\item\label{KettoC} 
If $\kappa_I\ni u>_\kappa v$ and $(\accentset{\bullet}a=)x\in \accentset{\bullet}H_u$ 
then $u>t$ hence we can consider $\mathbf X_{u_\star}$ in which
$(a_\star)^\bullet$ is in between $x$ and $y$.
Indeed, it follows from (\ref{kisebbD}) that
$x<(a_\star)^\bullet$
since 
$\kappa_I\ni u>_\kappa u_\star$,
$x\in\accentset{\bullet}H_u$,
and
$
\zeta_{u\to uu_\star}(x)  
=
\zeta_{u\to u}(x)  
\overset{(\ref{DEFzeta})}{=}
\varsigma_{u\to u}(a)  
\overset{\ref{IDes}}{=}
a
\overset{\ref{vS2}}{=} 
\varphi_{u\to u_\star}^{-1}(a_\star)
\overset{\ref{IDes}}{=} 
(\varsigma_{u\to u}\circ\varphi_{u\to u_\star}^{-1})(a_\star)
\overset{\ref{vS3}}{=}
\varsigma_{u_\star\to u}(a_\star)
\overset{(\ref{DEFzeta})}{=} 
\zeta_{u_\star\to u}((a_\star)^\bullet)
=
\zeta_{u_\star\to uu_\star}((a_\star)^\bullet)
$.
\end{enumerate}
\end{enumerate}
Summing up, $\mathfrak I^{\mathfrak{sl}}_{\mathfrak 0}$ is densifiable.
To see that $\mathfrak I^{\mathfrak{sl}}_{\mathfrak 0,\mathfrak{symm}}$ is densifiable, too, 
notice that since in Lemma~\ref{kjh986} the new element ($v^\star$ or $v_\star$) goes to the $\kappa_I$ partition, the extensions preserve idempotence symmetry.
Therefore, the extensions in the proof above are in $\mathfrak I^{\mathfrak{c}}_{\mathfrak 0,\mathfrak{symm}}$ if so is $\mathbf X$.
\end{proof}
Since 
$\aleph_0+\aleph_0=\aleph_0$,
an immediate corollary of Theorem~\ref{EzaTuTtIDENSIFIABLE} and Proposition~\ref{BaldiTeruiREF} is

\begin{corollary}\label{EzaTuTtIDENSIFIABLEchains}
Every is countable
$\mathbf X\in\mathfrak I^{\mathfrak{c}}_{\mathfrak 0}$ 
(resp. 
$\mathbf X\in\mathfrak I^{\mathfrak{c}}_{\mathfrak 0,\mathfrak{symm}}$)
can be embedded into a countable dense chain $\mathbf Y$ in $\mathfrak I^{\mathfrak{c}}_{\mathfrak 0}$
(resp. 
$\mathfrak I^{\mathfrak{c}}_{\mathfrak 0,\mathfrak{symm}}$).
\end{corollary}

\begin{remark}\label{TOPos} 
It is readily seen from (\ref{INNENfrom})--(\ref{kisebbD})
that a
nontrivial
odd or even involutive FL$_e$-algebra $\mathbf X$ is bounded if and only if its skeleton has a greatest element, it belongs to the $\kappa_I$ partition, and the corresponding layer group (hence also its subgroup) is the trivial (one-element) abelian $o$-group.
Indeed, assume $\mathbf X$ has a top element.
For $u<v$ and any element $x$ in $L_u$, 
$\zeta_{u\to uv}(x)$ is greater than $x$ by (\ref{kisebbAB}).
Therefore, the skeleton of $\mathbf X$ must have a greatest element, say $u$, and the top element must belong to $L_u$.
Abelian $o$-groups are known to be either trivial or unbounded. If $\bm{G}_u$ is unbounded then for any element $x$ in $L_u$ there exists a larger element in $\mathbf X$; namely, if $x\in G_u$ then any element in $G_u$ which is larger than $x$ in $G_u$ will also be larger in $\mathbf X$ by (\ref{INNENfrom}), and  
if $x\in \accentset{\bullet}H_u$ then $\zeta_{u\to u}(x)$ is larger than $x$ in $\mathbf X$ by (\ref{kisebbC}).
Therefore, $\bm{G}_u$ must be trivial, thus $G_u=\{u\}$, c.f.\ (\ref{IgyNeznekKi}). 
Since $\bm{G}_u$ is trivial, $u$ cannot be  in $\kappa_J$ by \ref{(G2)}. 
Since $u=t$ would imply,
using the maximality of $u$ in the skeleton, that $\mathbf X$ is trivial, $u\notin\kappa_o$ follows from \ref{(G1)}.
Therefore, $u$ must be in $\kappa_I$.
The other direction is straightforward as $u$ and $\accentset{\bullet}u$ are the top and the bottom elements of $\mathbf X$, respectively, shown by (\ref{kisebbAB}) and (\ref{kisebbD}), respectively.
\end{remark}

\begin{corollary}\label{EzaTuTtIDENSIFIABLEchainsBOUNDED}
Every is countable bounded
$\mathbf X\in\mathfrak I^{\mathfrak{c}}_{\mathfrak 0}$ 
(resp. 
$\mathbf X\in\mathfrak I^{\mathfrak{c}}_{\mathfrak 0,\mathfrak{symm}}$)
can be embedded into a countable bounded dense chain $\mathbf Y$ in $\mathfrak I^{\mathfrak{c}}_{\mathfrak 0}$
(resp. 
$\mathfrak I^{\mathfrak{c}}_{\mathfrak 0,\mathfrak{symm}}$) such that the embedding preserves the top and the bottom elements, too.
\end{corollary}
\begin{proof}
By Remark~\ref{TOPos}, there exists the top element of the skeleton of $\mathbf X$, it belongs to $\kappa_I$, and the corresponding layer group is trivial.
Since the embeddings in all the cases (\ref{EgyA}),  (\ref{standardDUMA}), (\ref{EgyC}),  (\ref{KettoA}), (\ref{KettoB}), and (\ref{KettoC})  
in the proof of Theorem~\ref{EzaTuTtIDENSIFIABLE} are done by coordinatewise identity mapings (c.f.\ the paragraph after Proposition~\ref{BaldiTeruiREF}), 
we only need to check that the constructions in these cases did not introduced new top and bottom elements in the FL$_e$-chain. 
By Remark~\ref{TOPos}, it suffices to check that the top element of its skeleton has not been changed.
Apart from (\ref{EgyA}) and (\ref{EgyC}), $\mathcal X_{v_\star}$ has been used, so in those cases the new element of the skeleton wend under another element. 
In (\ref{EgyA}) and (\ref{EgyC}), 
if $u>_\kappa v$ then the new element $v^\star$ goes in between them, whereas if $u=v$ then 
both $\zeta_{v\to v}(x)$ and $\zeta_{v\to v}(y)$ are in the $G_v$, they are different by (\ref{OurCondi1}), hence $\bm{G}_v$ is not trivial,
so $v$ cannot be the top element of the skeleton, and hence the new element goes between $v$ and the top element.
\end{proof}

Building upon Corollary~\ref{EzaTuTtIDENSIFIABLEchainsBOUNDED}, we are now prepared to prove

\begin{theorem}
The variety $\mathfrak I^{\mathfrak{sl}}_{\mathfrak 0}$ is strongly standard complete.
That is, Involutive Uninorm Logic with Fixed Point is strongly standard complete.
\end{theorem}
\begin{proof}
As for the MacNeille completion step, adapt the notation of Corollary~\ref{EzaTuTtIDENSIFIABLEchainsBOUNDED}.
Bounded countable dense chains have been shown to be isomorphic to $\mathbb Q\cap[0,1]$ by Cantor, hence $\mathbf X$ can be embedded (surjectively) into an odd involutive FL$_e$-chain over $\mathbb Q\cap[0,1]$.
Denoting its monoidal operation by $\te$,
$$
\gpont{a}{b}=\sup_{x\in\mathbb Q\cap[0,1], x<a}\sup_{y\in\mathbb Q\cap[0,1], y<b}\g{x}{y}
\ \ \ \ \mbox{(for $a,b\in[0,1]$)}
$$
results in a commutative, associative, residuated (equivalently left-continuous) extension $\tepont$ of $\te$ over $[0,1]$ (exactly like in \cite[Theorem 3.2]{JMstcompl}).
Obviously, being odd is inherited by $\tepont$.
Being involutive is inherited, too, since the residual complement operation being involutive is known to be equivalent to being strictly decreasing, 
and the strictly decreasing nature of the residual complement operation is clearly inherited using that $\mathbb Q\cap[0,1]$ is dense in $\mathbb R\cap[0,1]$.
\end{proof}

\begin{remark}
The property of idempotence symmetry is not inherited by a MacNeille completion, in general.
Indeed, the $o$-group $\mathbb Z\lex\mathbb R$ is in $\mathfrak I^{\mathfrak{c}}_{\mathfrak 0,\mathfrak{symm}}$.
However, it is not difficult to see that its MacNeille completion is $\PLPII{\mathbb Z}{\mathbb R}$ 
(see \cite[Definition 4.2]{JS_Hahn} for its definition or \cite[Example 3.1]{JS_Hahn}
for its analytic description),
and it is not in $\mathfrak I^{\mathfrak{c}}_{\mathfrak 0,\mathfrak{symm}}$.
\end{remark}

\subsection*{Acknowledgement}
This work has been supported by the NKFI-K-146701 grant.

\end{document}